\documentclass[a4paper,11pt]{amsart}

\usepackage{amsmath}
\usepackage{t1enc} 
\usepackage{epsf}
\usepackage{amssymb}
\usepackage{color}
\usepackage{graphicx}
\usepackage{bbm}

\newtheorem{theorem}{Theorem}[section]
\newtheorem{lem}[theorem]{Lemma}

\newtheorem{coro}[theorem]{Corollary}
\newtheorem{fact}[theorem]{Fact}

\theoremstyle{definition}
\newtheorem{defi}[theorem]{Definition}
\newtheorem{ex}[theorem]{Example}
\newtheorem{rem}[theorem]{Remark}

\numberwithin{equation}{section}
\numberwithin{theorem}{section}

\newdimen{\standardlabelwidth}
\standardlabelwidth-\standardlabelwidth
  \advance\standardlabelwidth by 2\normalparindent
\newcommand{\standardlabel}[1]{#1\kern\standardlabelwidth}

\begin{document}

\title[On $U$-polygons of class $c\geq 4$ in planar
  point sets]{On the existence of $U$-polygons\\ of class $c\geq 4$ in planar
  point sets}
\author{Christian Huck}
\address{Department of Mathematics and Statistics\\
  The Open University\\ Walton Hall\\ Milton Keynes\\ MK7 6AA\\
   United Kingdom}
\email{c.huck@open.ac.uk}
\thanks{This work was supported by EPSRC via Grant EP/D058465/1. It is a pleasure to thank the CRC 701 at the University of
  Bielefeld for support during a stay in June 2007, where part
  of the manuscript was written.}

\begin{abstract}
For a finite set $U$ of directions in the Euclidean plane, a convex
non-degenerate polygon $P$ is called a $U$-polygon if every
line parallel to a direction of $U$ that meets a vertex of $P$ also
meets another vertex of $P$. We characterize the numbers of edges of $U$-polygons of class $c\geq
4$ with all their vertices in certain subsets of the plane and derive
explicit results in the case of cyclotomic model sets.   
\end{abstract}

\maketitle

\section{Introduction}
\label{sec1}
The (discrete parallel)
  $X$-ray of a finite subset $F$ of the Euclidean plane in direction $u$ is the
corresponding line sum function that gives
  the numbers of points of $F$ on each line parallel to $u$. It was
  shown in~\cite[Proposition 4.6]{H1} that the convex subsets of an {\em algebraic Delone set} $\varLambda$ are determined by
their discrete parallel $X$-rays in the directions of a set $U$
  of at least two pairwise non-parallel 
$\varLambda$-directions (i.e., directions parallel to
  non-zero interpoint vectors of $\varLambda$) if and only if there is no $U$-polygon with
all its vertices in $\varLambda$. By~\cite[Lemma 4.5]{H1}, there always exists a $U$-polygon with
all its vertices in $\varLambda$ if $U$ is a set of at most three pairwise non-parallel 
$\varLambda$-directions. This leads to the question which
$U$-polygons exist with all their vertices in
$\varLambda$ for sets $U$ of four or more pairwise non-parallel 
$\varLambda$-directions. We refer the reader to~\cite{GG,HK,H0,H2,H1,H} for more on discrete
  tomography and~\cite{G} for the role of $U$-polygons in geometric
  tomography, where the $X$-ray of a compact convex set in a
  direction gives the lengths of all chords of the
  set in this direction. Dulio and Peri have introduced the notion
of {\em class} of a $U$-polygon and demonstrated that for planar
  {\em lattices} $L$ the numbers of edges of $U$-polygons of class $c\geq 4$ with
all their vertices in $L$ are precisely $8$ and $12$;
cf.~\cite[Theorem 12]{DP}. As a
first step beyond the case of planar lattices, this text
provides a generalization of this result to planar
sets that are non-degenerate in some sense and satisfy
a certain affinity condition on finite scales (Theorem~\ref{charac}). It turns out that, for
these sets $\varLambda$, the existence
of $U$-polygons of class $c\geq 4$ with
all their vertices in $\varLambda$ is equivalent to the existence of
certain {\em affinely regular polygons} with all their vertices in
$\varLambda$, a problem that was addressed in~\cite{H}. The obtained
characterization of numbers of vertices of $U$-polygons of class $c\geq 4$ with
all their vertices in $\varLambda$ can be expressed in terms of a simple inclusion of
real field extensions of $\mathbbm{Q}$ and particularly applies to
\emph{algebraic Delone sets}, thus including {\em cyclotomic model
  sets}, which form an important class of planar {\em mathematical
  quasicrystals}; cf.~\cite{B,BM}. For cyclotomic model
  sets $\varLambda$, the numbers of vertices of $U$-polygons of class $c\geq 4$ with
all their vertices in $\varLambda$ can be expressed by a
simple divisibility condition (Corollary~\ref{th1mod}). In particular,
the above result on lattice
$U$-polygons of class $c\geq 4$ by Dulio and Peri is contained as a special case (Corollary~\ref{cor2}(a)).    

\section{Definitions and preliminaries}
\label{sec2}

Natural numbers are always assumed to be positive and the set of
rational primes is denoted by $\mathcal{P}$. Primes $p\in\mathcal{P}$
for which the number $2p+1$ is prime as well are called \emph{Sophie
  Germain primes}. We denote by $\mathcal{P}_{\rm SG}$ the set of
Sophie Germain primes. The first few ones are
$$
2,3,5,11,23,29,41,53,83,89,113,131,173,179,191,233,239,\dots;$$ see entry A005384 of~\cite{Sl}
	   for further details. The group of units of a
given ring $R$ is denoted by $R^{\times}$. As usual, for a complex
number $z\in\mathbbm{C}$, $\vert z\vert$ denotes the complex absolute value $\vert z\vert=\sqrt{z\bar{z}}$, where $\bar{.}$ denotes the complex
conjugation. Occasionally, we
identify $\mathbbm{C}$ with $\mathbbm{R}^{2}$. The unit circle $\{x\in\mathbbm{R}^2\,|\,\vert x \vert =1\}$ in
$\mathbbm{R}^{2}$ is denoted by $\mathbb{S}^{1}$. Moreover, the elements of $\mathbb{S}^{1}$ are also called
{\em directions}. For a direction $u\in
\mathbb{S}^{1}$, the {\em angle between $u$ and the positive real axis} is understood to be
the unique angle $\theta\in [0,\pi)$ with the property that a
rotation of $1\in\mathbbm{C}$ by $\theta$ in counter-clockwise order is a
direction parallel to $u$. For $r>0$ and $x\in\mathbbm{R}^{2}$,
$B_{r}(x)$ denotes the open ball of radius $r$ about $x$. A subset $\varLambda$ of the plane is called {\em uniformly discrete} if there is a radius
$r>0$ such that every ball $B_{r}(x)$ with $x\in\mathbbm{R}^{2}$ contains at most one point of
  $\varLambda$. Further, $\varLambda$ is called {\em relatively dense} if there is a radius $R>0$
  such that every ball $B_{R}(x)$ with 
  $x\in\mathbbm{R}^{2}$ contains at least one point of $\varLambda$. $\varLambda$ is called a {\em Delone set} if it is both uniformly
  discrete and relatively dense. A direction
$u\in\mathbb{S}^{1}$ is called a $\varLambda${\em-direction} if it is
parallel to a non-zero element of the difference set
$\varLambda-\varLambda$ of $\varLambda$. Further, a bounded subset $C$ of $\varLambda$ is called a {\em convex subset of} $\varLambda$ if its
convex hull contains no new points of $\varLambda$. A {\em
  non-singular affine transformation} of the Euclidean plane is given
by $z \mapsto Az+t$, where $A\in\operatorname{GL}(2,\mathbbm{R})$ and
$t\in \mathbbm{R}^{2}$. Further, recall that a {\em homothety} of the Euclidean plane is given by $z \mapsto \lambda z + t$, where
$\lambda \in \mathbbm{R}$ is positive and $t\in \mathbbm{R}^{2}$. A {\em convex polygon} is the convex hull of a
finite set of points in $\mathbbm{R}^2$. For a subset $S \subset
\mathbbm{R}^2$, a {\em polygon in} $S$ is a convex polygon with all
vertices in $S$. A {\em regular polygon} is always assumed to be
planar, non-degenerate and convex. An {\em affinely regular polygon}
is a non-singular affine image of a regular polygon. In particular, it
must have at least $3$ vertices. Let $U\subset \mathbb{S}^{1}$ be a
finite set of pairwise non-parallel directions. A non-degenerate convex polygon $P$ is
called a $U${\em -polygon} if it has the property that whenever $v$ is
a vertex of $P$ and $u\in U$, the line $\ell_{u}^{v}$ in the plane in
direction $u$ which passes through $v$ also meets another vertex $v'$
of $P$. Then, every direction of $U$ is parallel to one of the edges of
$P$; cf.~\cite[Lemma 5(i)]{DP}. Further, one can easily see that a $U$-polygon has $2m$ edges, where
$m\geq\operatorname{card}(U)$. For example, an affinely regular polygon with an
even number of vertices is a $U$-polygon if and only if each direction
of $U$ is parallel to one of its edges. The following notion of
{\em class} of a $U$-polygon was introduced by Dulio and Peri;
cf.~\cite[Definition 1]{DP}. For $0<c\leq\operatorname{card}(U)$, a
$U$-polygon $P$ is said to be of {\em class} $c$ with respect to $U$ if $c$ is the
maximal number of consecutive edges of $P$ whose directions belong to
$U$. 

\begin{defi}
For a subset $\varLambda\subset\mathbbm{C}$, we denote by $\mathbbm{K}_{\varLambda}$ the intermediate field of $\mathbbm{C}/\mathbbm{Q}$ that is given by 
$$
\mathbbm{K}_{\varLambda}\,:=\,\mathbbm{Q}\left(\big(\varLambda-\varLambda\big)\cup\big(\overline{\varLambda-\varLambda}\big)\right)\,.
$$
Further, we set 
$
\mathbbm{k}_{\varLambda}:=\mathbbm{K}_{\varLambda}\cap\mathbbm{R}$,
the maximal real subfield of $\mathbbm{K}_{\varLambda}$.
\end{defi} 

For $n\in
\mathbbm{N}$, we always let $\zeta_n := e^{2\pi i/n}$, as a specific
choice for a
primitive $n$th root of unity in $\mathbbm{C}$. Denoting by $\phi$ Euler's
totient function, one has the following standard result for the
$n$th cyclotomic field $\mathbbm{Q}(\zeta_n)$.

\begin{fact}[Gau\ss]{\rm \cite[Theorem 2.5]{Wa}}\label{gau}
$[\mathbbm{Q}(\zeta_n) : \mathbbm{Q}] = \phi(n)$. The field extension $\mathbbm{K}_{n}/ \mathbbm{Q}$
is a Galois extension with Abelian Galois group $G(\mathbbm{Q}(\zeta_n)/
\mathbbm{Q}) \simeq (\mathbbm{Z} / n\mathbbm{Z})^{\times}$,
where $a\, (\textnormal{mod}\, n)$ corresponds to the automorphism given by\/ $\zeta_n \mapsto \zeta_n^{a}$. 
\end{fact}

It is well known that
$\mathbbm{Q}(\zeta_n+\bar{\zeta}_{n})$ is the maximal real subfield of
$\mathbbm{Q}(\zeta_n)$ and is of degree $\phi(n)/2$ over $\mathbbm{Q}$; see~\cite{Wa}. Throughout this text, we shall use the
notation $$\mathbbm{K}_{n}=\mathbbm{Q}(\zeta_n),\;
\mathbbm{k}_{n}=\mathbbm{Q}(\zeta_n+\bar{\zeta}_{n}),\; \mathcal{O}_{n}=\mathbbm{Z}[\zeta_n],\; \thinspace\scriptstyle{\mathcal{O}}\displaystyle_{n}
=\mathbbm{Z}[\zeta_n+\bar{\zeta}_{n}]\,.$$ 
Note that that $\mathcal{O}_{n}$ and $\thinspace\scriptstyle{\mathcal{O}}\displaystyle_{n}$ are the rings of
integers in $\mathbbm{K}_{n}$ and $\mathbbm{k}_{n}$, respectively;
cf.~\cite[Theorem 2.6 and Proposition 2.16]{Wa}. For $n$ odd,
one has $\phi(2n)=\phi(n)$ by the multiplicativity of the arithmetic
function $\phi$ and thus
 $\mathbbm{K}_{n}=\mathbbm{K}_{2n}$; cf. Fact~\ref{gau}.

\begin{defi}\label{algdeldef}
For a set $\varLambda\subset\mathbbm{R}^2$, we define the following properties:
\begin{eqnarray*}
\mbox{{\rm (Alg)}}&\hphantom{a}&\left[\mathbbm{K}_{\varLambda}:\mathbbm{Q}\right]<\infty\,.\\
\mbox{{\rm (Aff)}}&\hphantom{a}&\mbox{For all finite subsets $F$ of
  $\mathbbm{K}_{\varLambda}$, there is a non-singular affine}\\&&
\mbox{transformation $\Psi$ of the plane such that
$\Psi(F)\subset \varLambda$\,.}\\
\mbox{{\rm (Hom)}}&\hphantom{a}&\mbox{For all finite subsets $F$ of
  $\mathbbm{K}_{\varLambda}$, there is a homothety $h$ of the}\\&&
\mbox{plane such that
$h(F)\subset \varLambda$\,.}
\end{eqnarray*}
Moreover, we call $\varLambda$ \emph{degenerate} if and only if\/ 
$\mathbbm{K}_{\varLambda}$ is a subfield of $\mathbbm{R}$.
\end{defi} 

\begin{rem}\label{mrs}
For any
non-degenerate $\varLambda\subset\mathbbm{R}^2$, the field $\mathbbm{K}_{\varLambda}$ is a
complex extension of $\mathbbm{Q}$. Trivially, property {\rm (Hom)} implies property {\rm (Aff)}. If $\varLambda$ satisfies property {\rm (Alg)}, then one has
$\left[\mathbbm{k}_{\varLambda}:\mathbbm{Q}\right]<\infty$, meaning that $\mathbbm{k}_{\varLambda}$ is a real algebraic number field.
\end{rem}

We need the following result of Darboux~\cite{D} on
second mid-point polygons, where the {\em midpoint polygon} $M(P)$ of a
convex polygon $P$ is the convex polygon whose vertices are the
midpoints of the edges of $P$; compare also~\cite[Lemma 5]{GM} or~\cite[Lemma
  1.2.9]{G}.  

\begin{fact}\label{darboux}
Let $P_0$ be a convex $n$-gon in $\mathbbm{R}^2$ with centroid at the
origin. For each $k\in\mathbbm{N}$, define
$P_k:=\operatorname{sec}(\pi /n)M(P_{k-1})$. Then the sequence
$(P_{2k})_{k=0}^{\infty}$ converges in the Hausdorff metric to an
affinely regular polygon.
\end{fact}

If, in the situation of Fact~\ref{darboux}, $P_0$ is a $U$-polygon of
class $c$, then, for all $k$, $P_{2k}$ is a $U$-polygon of
class $c$, whence also $R:=\lim_{k\rightarrow \infty}P_{2k}$ is a $U$-polygon of
class $c$. This proves the next

\begin{lem}\label{uaffine}
Let $U\subset \mathbb{S}^{1}$ be a finite
set of directions and let $0<c\leq\operatorname{card}(U)$. Then, there
exists a $U$-polygon of class $c$ if and only if there is an affinely
regular $U$-polygon of class $c$.
\end{lem}

Let $(t_1,t_2,t_3,t_4)$ be an ordered tuple of four distinct
elements of the set $\mathbbm{R}\cup\{\infty\}$. Then, its {\em cross ratio}
$\langle t_1,t_2,t_3,t_4\rangle$ is defined by
$$
\langle t_1,t_2,t_3,t_4\rangle := \frac{(t_3 - t_1)(t_4 - t_2)}{(t_3 - t_2)(t_4 - t_1)}\,,
$$
with the usual conventions if one of the $t_i$ equals $\infty$, thus
$\langle t_1,t_2,t_3,t_4\rangle\in \mathbbm{R}$. The following
property of cross ratios of slopes $s_z$ of elements $z\in\mathbbm{R}^2\setminus\{0\}$ is standard.

\begin{fact}\label{crossratio}
Let $z_{j}\in \mathbbm{R}^2\setminus\{0\}$, $j\in \{1,\dots,4\}$, be four pairwise
non-parallel elements of the Euclidean plane and let $A\in\operatorname{GL}(2,\mathbbm{R})$. Then, one has
$$
\langle s_{z_{1}},s_{z_{2}},s_{z_{3}},s_{z_{4}}\rangle = \langle s_{A z_1},s_{A z_2},s_{A
  z_3},s_{A z_4}\rangle\,.
$$  
\end{fact}

\begin{lem}{\rm \cite[Fact 4.7]{H1}}\label{cr}
For a set $\varLambda\subset\mathbbm{R}^2$, the
cross ratio of slopes of four pairwise non-parallel
$\varLambda$-directions is an element of the field $\mathbbm{k}_{\varLambda}$.
\end{lem}

\section{The characterization}
\label{sec3}

\begin{theorem}\label{charac}
For a non-degenerate subset $\varLambda$ of the plane with property
{\rm (Aff)} and an even number $m\geq 8$, the
following statements are equivalent:
\begin{itemize}
\item[{\rm (i)}]
There is a $U$-polygon of class $c\geq 4$ in $\varLambda$ with $m$ edges.
\item[{\rm (ii)}]
There is an affinely regular $U$-polygon of class $c\geq 4$ with $m$
edges for a set $U$ of $\varLambda$-directions.
\item[{\rm (iii)}]
$\mathbbm{k}_{m/2}\subset\mathbbm{k}_{\varLambda}$.
\item[{\rm (iv)}]
There is an affinely regular polygon in $\varLambda$ with $\operatorname{lcm}(m/2,2)$ edges.
\end{itemize}
If $\varLambda$ additionally fulfils property {\rm (Alg)}, then the above assertions only hold for finitely many values of $m$. 
\end{theorem}
\begin{proof}
Direction (i) $\Rightarrow$ (ii) immediately follows from
Lemma~\ref{uaffine}. For direction (ii) $\Rightarrow$ 
(iii), let $P$ be an affinely regular $U$-polygon of class $c\geq 4$ with $m$
edges for a set $U$ of $\varLambda$-directions. There is then a
non-singular affine transformation $\Psi$ of the plane such that $R_{m} = \Psi(P)$ is
a regular $m$-gon. Since $P$ is a $U$-polygon of class $c\geq 4$ for a set $U$ of
$\varLambda$-directions and since, by Fact~\ref{crossratio}, the cross ratio of slopes of
directions of edges is preserved by non-singular
affine transformations, there are four consecutive edges of $R_m$
whose cross ratio $q$ of slopes of their directions, say arranged in order of
increasing angle with the positive real axis, is an element of
$\mathbbm{k}_{\varLambda}$; cf.~Lemma~\ref{cr}. Applying a suitable
rotation, if necessary, we may assume
that these directions are given in complex form by
$1,\zeta_m,\zeta_m^2$ and $\zeta_m^3$; cf.~Fact~\ref{crossratio} again. Using
$\sin(\theta)=-e^{-i\theta}(1-e^{2i\theta})/2i$, one easily calculates
that 
\begin{eqnarray*}
q&=&\frac{(\tan (\frac{3 \pi}{m/2})-
  \tan (\frac{ \pi}{m/2}))(\tan (\frac{2 \pi}{m/2})- \tan (\frac{0
    \pi}{m/2}))}{(\tan (\frac{3 \pi}{m/2})- \tan (\frac{0
    \pi}{m/2}))(\tan (\frac{2 \pi}{m/2})- \tan (\frac{
    \pi}{m/2}))}=\frac{\sin(\frac{2\pi}{m/2})\sin(\frac{2\pi}{m/2})}{\sin(\frac{\pi}{m/2})\sin(\frac{3\pi}{m/2})}\\&=&\frac{(1-\zeta_{m/2}^2)(1-\zeta_{m/2}^2)}{(1-\zeta_{m/2})(1-\zeta_{m/2}^3)}=\frac{2+\zeta_{m/2}+\bar{\zeta}_{m/2}}{1+\zeta_{m/2}+\bar{\zeta}_{m/2}}\in\mathbbm{k}_{\varLambda}\,.
\end{eqnarray*}
This implies that
$$
\frac{q}{q-1}-2=\zeta_{m/2}+\bar{\zeta}_{m/2}\in\mathbbm{k}_{\varLambda}\,,
$$
the latter being equivalent to (iii). Direction (iii) $\Rightarrow$ 
(iv) is an immediate consequence of~\cite[Theorem
  3.3]{H} in conjunction with the identity $\mathbbm{k}_{m/2}=\mathbbm{k}_{\operatorname{lcm}(m/2,2)}$. For direction (iv) $\Rightarrow$ 
(i), assume first that $m/2$ is odd. Here, we are done since every
affinely regular polygon in $\varLambda$ with
$\operatorname{lcm}(m/2,2)=m$ edges is a $U$-polygon of class $c=
m/2$ with respect to any set $U$ of directions parallel to $m/2$ consecutive of
its edges. If $m/2$ is even, there is an affinely regular polygon $P$ in $\varLambda$ with
$\operatorname{lcm}(m/2,2)=m/2$ edges. Attach $m/2$ translates of $P$
edge-to-edge
to $P$ in the obvious way and consider the convex hull $P'$ of the resulting point
set. Clearly, $P'$ is a $U'$-polygon in $\mathbbm{K}_{\varLambda}$ of class
$c=\operatorname{card}(U')$ with $m$
edges, where $U'$ consists of the $m/2$ pairwise non-parallel $\varLambda$-directions
 given by the edges and diagonals of $P$. By property~(Aff), there is
 a non-singular affine transformation $\Psi$ of the plane such that
 $\Psi(P')$ is a polygon in $\varLambda$. Then, $\Psi(P')$ is a
 $U$-polygon of class $c=\operatorname{card}(U)$ in
 $\varLambda$ with $m$ edges, where $U$ is a set of $m/2$ pairwise
 non-parallel $\varLambda$-directions parallel to the elements of
 $\Psi(U')$. Assertion (i) follows. If $\varLambda$ additionally has
 property~(Alg), then $\mathbbm{k}_{\varLambda}$ is an algebraic
 number field by
 Remark~\ref{mrs}. Thus, the field extension
 $\mathbbm{k}_{\varLambda}/\mathbbm{Q}$ has only finitely many
 intermediate fields and the assertion follows from condition (iii) in
 conjunction with~\cite[Corollary 2.7, Remark~2.8 and Lemma~2.9]{H}.\end{proof}

\begin{coro}\label{numcor}
Let $\mathbbm{L}$ be a complex algebraic number field with
$\overline{\mathbbm{L}}=\mathbbm{L}$ and let
$\mathcal{O}_{\mathbbm{L}}$ be the ring of integers in
$\mathbbm{L}$. Let $\varLambda$ be a translate of $\mathbbm{L}$ or a
translate of $\mathcal{O}_{\mathbbm{L}}$. Further, let $m\geq 8$ be
an even number. Denoting the maximal real subfield of $\mathbbm{L}$ by $\mathbbm{l}$, the following statements are equivalent:
\begin{itemize}
\item[{\rm (i)}]
There is a $U$-polygon of class $c\geq 4$ in $\varLambda$ with $m$ edges.
\item[{\rm (ii)}]
There is an affinely regular $U$-polygon of class $c\geq 4$ with $m$
edges for a set $U$ of $\varLambda$-directions.
\item[{\rm (iii)}]
$\mathbbm{k}_{m/2}\subset\mathbbm{l}$.
\item[{\rm (iv)}]
There is an affinely regular polygon in $\varLambda$ with $\operatorname{lcm}(m/2,2)$ edges.
\end{itemize}
Additionally, the above assertions only hold for finitely many values of $m$.
\end{coro}
\begin{proof}
This follows immediately from Theorem~\ref{charac} in conjunction with
the fact that $\varLambda$ has properties (Aff) and (Alg) with
$\mathbbm{K}_{\varLambda}=\mathbbm{L}$; cf.~\cite[Section 3]{H}.
\end{proof}

\begin{rem}
In particular, Corollary~\ref{numcor} applies to translates of complex
cyclotomic fields and their rings of integers, respectively, with
$\mathbbm{l}=\mathbbm{k}_n$ for a suitable $n\geq 3$;
cf.~Fact~\ref{gau} and also compare the equivalences of Corollary~\ref{th1mod} below. 
\end{rem}

\section{Application to cyclotomic model sets}
\label{sec4}

Delone subsets of the plane satisfying
properties~(Alg) and (Hom) were introduced as
\emph{algebraic Delone sets} in~\cite[Definition 4.1]{H1}. Note that
algebraic Delone sets are always non-degenerate, since this is true
for all relatively dense subsets of the plane. Examples of algebraic
Delone sets are the so-called \emph{cyclotomic
  model sets} $\varLambda$; cf.~\cite[Proposition 4.31]{H1}. By
definition, any cyclotomic model set
$\varLambda$ is contained in a translate of $\mathcal{O}_n$, where $n\geq 3$,
in which case the $\mathbbm{Z}$-module $\mathcal{O}_n$ is called the \emph{underlying
  $\mathbbm{Z}$-module} of $\varLambda$. More precisely, for $n\geq 3$, let $.^{\star}\!:\, \mathcal{O}_n\rightarrow
(\mathbbm{R}^2)^{\phi(n)/2-1}
$
be any map of the form
$$
z\mapsto
(\sigma_{2}(z),\dots,\sigma_{\phi(n)/2}(z))\,,
$$
where the
set $\{\sigma_{2},\dots,\sigma_{\phi(n)/2}\}$ arises from $G(\mathbbm{K}_{n}/ \mathbbm{Q})\setminus\{\operatorname{id},\bar{.}\}$ by choosing exactly one automorphism
from each pair of complex conjugate ones;
cf.~Fact~\ref{gau}. Then, for any such choice, each translate
of the set $\{z\in \mathcal{O}_n\,|\,z^{\star}\in W\}$, where $W\subset(\mathbbm{R}^2)^{\phi(n)/2-1}$ is a
sufficiently `nice' set with non-empty
interior and compact closure,  
is a cyclotomic model set with underlying $\mathbbm{Z}$-module $\mathcal{O}_n$; cf.~\cite{H0,H2,H1,H} for more details and
properties of (cyclotomic) model sets. Since $\mathcal{O}_n=\mathcal{O}_{2n}$ for odd $n$, we might restrict
ourselves to values $n \not\equiv 2
\;(\operatorname{mod} 4)$ when dealing with
cyclotomic model sets with underlying $\mathbbm{Z}$-module
$\mathcal{O}_n$. With the exception of the
crystallographic cases of translates of the square lattice $\mathcal{O}_4$ and
translates of the triangular lattice $\mathcal{O}_3$, cyclotomic model sets
are aperiodic (they have no non-zero translational symmetries) and have
long-range order; cf.~\cite[Remark 4.23]{H1}. Well-known
examples of cyclotomic model sets with underlying $\mathbbm{Z}$-module $\mathcal{O}_n$
are the vertex sets of aperiodic tilings of the plane like the
Ammann-Beenker tiling~\cite{am,bj,ga} ($n=8$), the T\"ubingen triangle
tiling~\cite{bk1,bk2} ($n=5$) and the shield tiling~\cite{ga}
($n=12$); cf.~Figure~\ref{fig:tilingupolygon} for an illustration. For definitions of the above vertex sets
of aperiodic tilings of the plane in algebraic terms, we refer the
reader to~\cite[Section 1.2.3.2]{H2} or~\cite{H0}. As an immediate
consequence of Theorem~\ref{charac} in conjunction
with~\cite[Corollary 4.1]{H} and the identity $\mathbbm{k}_{m/2}=\mathbbm{k}_{\operatorname{lcm}(m/2,2)}$, one obtains the following

\begin{coro}\label{th1mod}
Let $m,n\in \mathbbm{N}$ with $m\geq 8$ an even number and $n\geq 3$. Further, let $\varLambda$ be a cyclotomic model set with underlying $\mathbbm{Z}$-module $\mathcal{O}_n$. The following statements are equivalent:
\begin{itemize}
\item[{\rm (i)}]
There is a $U$-polygon of class $c\geq 4$ in $\varLambda$ with $m$ edges.
\item[{\rm (ii)}]
There is an affinely regular $U$-polygon of class $c\geq 4$ with $m$
edges for a set $U$ of $\varLambda$-directions.
\item[{\rm (iii)}]
$\mathbbm{k}_{m/2}\subset\mathbbm{k}_{\varLambda}$.
\item[{\rm (iv)}]
There is an affinely regular polygon in $\varLambda$ with $\operatorname{lcm}(m/2,2)$ edges.
\item[{\rm (v)}]
$\mathbbm{k}_{m/2}\subset\mathbbm{k}_{n}$.
\item[{\rm (vi)}]
$m \in \{8,12\}$, or $\mathbbm{K}_{m/2}\subset \mathbbm{K}_{n}$.
\item[{\rm (vii)}]
$m \in \{8,12\}$, or $m|2n$, or $m=4d$ with $d$ an odd divisor of $n$.
\item[{\rm (viii)}]
$m \in \{8,12\}$, or $\mathcal{O}_{m/2}\subset \mathcal{O}_{n}$.
\item[{\rm (ix)}]
$\thinspace\scriptstyle{\mathcal{O}}\displaystyle_{m/2}\subset\thinspace\scriptstyle{\mathcal{O}}\displaystyle_{n}$.
\end{itemize}
\end{coro}

\begin{rem}
Combining Corollary~\ref{th1mod} and Fact~\ref{crossratio}, one sees
that the cross ratios of slopes of
directions of edges of $U$-polygons of class $c\geq 4$ in cyclotomic
model sets $\varLambda$, say arranged in order of
increasing angle with the positive real axis, easily follow from a
direct computation with a finite number of regular polygons;
cf.~\cite{GG,BDNP} for deep insights into this in the case of planar lattices.\end{rem}

The following
consequence follows immediately from Corollary~\ref{th1mod} in
conjunction with~\cite[Corollary 4.2]{H}. Restricted to values $n \not\equiv 2
\;(\operatorname{mod} 4)$, it deals with the two cases where the degree $\phi(n)/2$ of
$\mathbbm{k}_{n}$ over $\mathbbm{Q}$ is either $1$ or a prime number
$p\in\mathcal{P}$; cf.~\cite[Lemma 2.10]{H}.

\begin{coro}\label{cor2}
Let $m,n\in\mathbbm{N}$ with $m\geq 8$ an even number and $n\geq 3$. Further, let $\varLambda$ be a cyclotomic model set with underlying $\mathbbm{Z}$-module $\mathcal{O}_n$. Then, one has:
\begin{itemize}
\item[{\rm (a)}] If $n\in\{3,4\}$, there is a $U$-polygon of class $c\geq 4$ in $\varLambda$ with $m$ edges if and only if $m \in \{8,12\}$.  
\item[{\rm (b)}] If $n\in\{8,9,12\} \cup \{2p+1\, | \, p \in
  \mathcal{P}_{\rm SG}\}$, there is a $U$-polygon of class $c\geq 4$
  in $\varLambda$ with $m$ edges if and only if 
$$\left\{
\begin{array}{ll}
m \in \{8,12,2n\}, & \mbox{if $n=8$ or $n=12$,}\\
m \in \{8,12,2n,4n\}, & \mbox{otherwise.}
\end{array}\right.
$$ 
\end{itemize}
\end{coro}

\begin{figure}
\centerline{\epsfysize=0.57\textwidth\epsfbox{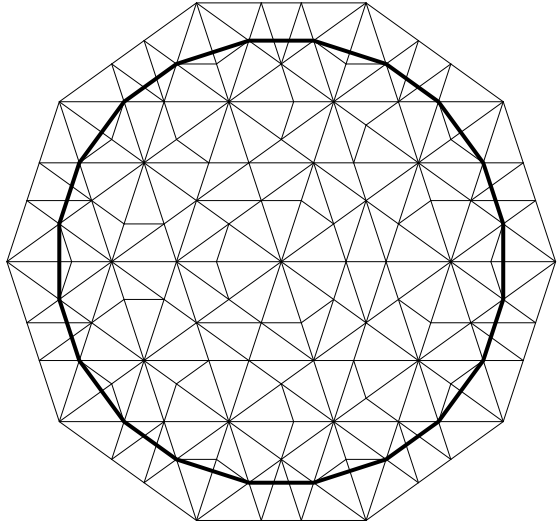}}
\caption{An $U$-icosagon of class $c=\operatorname{card}(U)=10$ in
  the vertex set $\varLambda_{\rm TTT}$ of the
  T\"ubingen triangle tiling with respect to the set $U$ of
  $\varLambda_{\rm TTT}$-directions
 given by the edges and diagonals of the central regular decagon.}
\label{fig:tilingupolygon}
\end{figure}

\begin{ex}
As mentioned above, the vertex set $\varLambda_{\rm TTT}$ of the T\"ubingen triangle
tiling is a cyclotomic model set with underlying $\mathbbm{Z}$-module
$\mathcal{O}_5$. By Corollary~\ref{cor2} there is a
$U$-polygon of class $c\geq 4$ in $\varLambda_{\rm TTT}$ with $m$
edges if and only if $m\in\{8,10,12,20\}$; see
Figure~\ref{fig:tilingupolygon} for an $U$-icosagon of class $c=10$ in
$\varLambda_{\rm TTT}$.
\end{ex}

\section*{Acknowledgements}
It is a pleasure to thank Michael Baake, Richard J. Gardner and Uwe
Grimm for their cooperation and for useful hints on the manuscript.


%


\begin{thebibliography}{99}


\bibitem{am} R. Ammann, B. Gr\"unbaum, G. C. Shephard, Aperiodic
  tiles, {\em Discrete Comput. Geom.} 8 (1992) 1--25. 

\bibitem{B} M. Baake, A guide to mathematical quasicrystals,
  in: J.-B. Suck, M. Schreiber, P. H\"aussler (Eds.), {\em Quasicrystals. An Introduction to Structure, Physical Properties,
    and Applications},  Springer, Berlin, 2002, pp. 17--48; \texttt{arXiv:math-ph/9901014v1
}  

\bibitem{BG2}  M. Baake, P. Gritzmann, C. Huck, B. Langfeld, K. Lord, Discrete tomography of planar model sets,
  {\em Acta Crystallogr. A} 62 (2006) 419--433; \texttt{arXiv:math/0609393v1 [math.MG]} 

\bibitem{bj} M. Baake, D. Joseph, Ideal and defective vertex configurations in the planar octagonal quasilattice, {\em Phys. Rev. B} 42 (1990) 8091--8102.

\bibitem{bk1} M. Baake, P. Kramer, M. Schlottmann, D. Zeidler, The triangle pattern -- a new quasiperiodic tiling with
    fivefold symmetry, {\em Mod. Phys. Lett. B} 4 (1990) 249--258.

\bibitem{bk2} M. Baake, P. Kramer, M. Schlottmann, D. Zeidler, Planar patterns with fivefold symmetry as sections of periodic
    structures in 4-space, {\em Int. J. Mod. Phys. B} 4 (1990) 2217--2268.

\bibitem{BM} M. Baake \and R. V. Moody (Eds.), {\em Directions in Mathematical Quasicrystals}, CRM Monograph Series, vol. 13, AMS, Providence, RI, 2000.

\bibitem{BDNP}
E. Barcucci, A. Del Lungo, M. Nivat, R. Pinzani, X-rays characterizing
some classes of discrete sets, {\em Linear Algebra Appl.} 339 (2001) 3--21.

\bibitem{D} M. G. Darboux, Sur un probl\`eme de g\'eom\'etrie \'el\'ementaire, {\em Bull. Sci. Math.} 2 (1878) 298--304.


\bibitem{DP} P. Dulio, C. Peri, On the geometric structure of lattice $U$-polygons, {\em Discrete Math.} 307 (2007) 2330--2340.

\bibitem{G} R. J. Gardner, {\em Geometric Tomography}, 2nd ed., Cambridge University Press, New York, 2006.

\bibitem{GG} R. J. Gardner, P. Gritzmann, Discrete tomography: determination of finite sets by X-rays, {\em Trans. Amer. Math. Soc.} 349 (1997) 2271--2295.

 \bibitem{GM} R. J. Gardner \and P. McMullen, On Hammer's X-ray problem, {\em J. London Math. Soc.} (2) 21 (1980) 171--175.

\bibitem{ga} F. G\"ahler, Matching rules for quasicrystals: the
  composition-decomposition method, {\em J. Non-Cryst. Solids}
  153-154 (1993) 160--164.

\bibitem{HK} G. T. Herman, A. Kuba (Eds.), {\em Discrete Tomography: Foundations, Algorithms, and Applications},
Birkh\"auser, Boston, 1999.

\bibitem{H0} C. Huck, Uniqueness in discrete tomography of planar model sets, notes (2007); \texttt{arXiv:math/0701141v2 [math.MG]} 

\bibitem{H2} C. Huck, {\em Discrete Tomography of Delone Sets with Long-Range Order}, Ph.D. thesis (Universit\"at Bielefeld), Logos Verlag, Berlin, 2007. 

\bibitem{H1} C. Huck, Uniqueness in discrete tomography of Delone sets with long-range order, submitted; \texttt{arXiv:0711.4525v2 [math.MG]}

\bibitem{H} C. Huck, A note on affinely regular polygons,
  {\em European J. Combin.}, in press; \texttt{arXiv:0801.3218v1 [math.MG]}

\bibitem{Sl} N. J. A. Sloane (ed.), {\em The Online Encyclopedia of
    Integer Sequences}, published electronically at \texttt{http://www.research.att.com/\symbol{126}njas/sequences/}.

\bibitem{Wa} L. C. Washington, {\em Introduction to Cyclotomic Fields}, 2nd ed., Springer, New York, 1997.

\end{thebibliography}
\end{document}